\documentclass[a4paper,11pt]{article}

\usepackage[utf8]{inputenc}
\usepackage[british]{babel}

\usepackage[mono=false]{libertine}
\usepackage[T1]{fontenc}

\usepackage{amsthm}
\usepackage[cmintegrals,libertine]{newtxmath}
\usepackage[cal=euler, scr=boondoxo]{mathalfa}
\useosf

\usepackage{microtype}

\usepackage{numprint}

\usepackage[margin=2.5cm]{geometry}
\usepackage{multirow}

\usepackage{tikz}

\usepackage{graphicx}
	\graphicspath{{./images/}}

\usepackage[font=sf]{caption}

\usepackage{hyperref}
\hypersetup{
	colorlinks=true,
		citecolor=blue!60!black,
		linkcolor=red!60!black,
		urlcolor=green!40!black,
		filecolor=yellow!50!black,
	breaklinks=true,
	pdfpagemode=UseNone,
	bookmarksopen=false,
}

\usepackage{enumitem}

\newtheorem{thm}{Theorem}
\newtheorem{lem}{Lemma}
\newtheorem{prop}{Proposition}

\theoremstyle{definition}

\newcommand{\ensembles}[1]{\mathbb{#1}}
	
	\newcommand{\Z}{\ensembles{Z}}
	\newcommand{\R}{\ensembles{R}}

\newcommand{\ind}[1]{\mathbf{1}_{\{#1\}}}
	\renewcommand{\P}{\ensembles{P}}
	\newcommand{\E}{\ensembles{E}}

\renewcommand{\Pr}[1]{\P\left(#1\right)}

\newcommand{\Es}[1]{\E\left[#1\right]}

\renewcommand{\d}{\mathrm{d}}
\newcommand{\ex}{\mathrm{e}}

\newcommand{\q}{\mathbf{q}}
\newcommand{\e}{\overline{\mathfrak{e}}}
\renewcommand{\root}{\mathfrak{R}}

\newcommand{\dgr}{d_{\mathrm{gr}}}

\newcommand{\hBall}{\overline{\mathrm{Ball}}}

        \newcommand{\map}{\mathfrak{m}}

\newcommand{\Map}{\mathfrak{M}_\infty}

\newcommand{\Alg}{\mathcal{A}}

\newcommand{\cv}[1][n]{\enskip\xrightarrow[#1 \to \infty]{}\enskip}
\newcommand{\cvdist}[1][n]{\enskip\xrightarrow[#1 \to \infty]{(d)}\enskip}

\DeclareSymbolFont{extraup}{U}{zavm}{m}{n}
\DeclareMathSymbol{\vardspade}{\mathalpha}{extraup}{81}
\DeclareMathSymbol{\varheart}{\mathalpha}{extraup}{86}
\DeclareMathSymbol{\vardiamond}{\mathalpha}{extraup}{87}
\DeclareMathSymbol{\varclub}{\mathalpha}{extraup}{84}

\makeatletter
\renewcommand*{\@fnsymbol}[1]{\ensuremath{\ifcase#1\or  \vardspade \or \varclub \or \vardiamond \or \varheart \or
   \mathsection\or \mathparagraph\or \|\or **\or \dagger\dagger
   \or \ddagger\ddagger \else\@ctrerr\fi}}
\makeatother

\author{
	Nicolas \textsc{Curien}\thanks{D\'epartement de Math\'ematiques, Univ. Paris-Sud, Universit\'e Paris-Saclay and IUF.\hfill  \href{mailto:nicolas.curien@gmail.com}{\texttt{nicolas.curien@gmail.com}}} 
\qquad\&\qquad
	Cyril \textsc{Marzouk}\thanks{D\'epartement de Math\'ematiques, Univ. Paris-Sud, Universit\'e Paris-Saclay and FMJH.\hfill  \href{mailto:cyril.marzouk@u-psud.fr}{\texttt{cyril.marzouk@u-psud.fr}}}
}

\title{How fast planar maps get swallowed by a peeling process}

\begin{document}

\maketitle

\begin{abstract}
The peeling process is an algorithmic procedure that discovers a random planar map step by step. In generic cases such as the UIPT or the UIPQ, it is known~\cite{Curien-Le_Gall:Scaling_limits_for_the_peeling_process_on_random_maps} that any peeling process will eventually discover the whole map. In this paper we study the probability that the origin is not swallowed by the peeling process until time $n$ and show it decays at least as $n^{-2c/3}$ where 
\[c \approx 0.1283123514178324542367448657387285493314266204833984375...\]
is defined via an integral equation derived using the Lamperti representation of the spectrally negative $3/2$-stable Lévy process conditioned to remain positive~\cite{Chaumont-Kyprianou-Pardo:Some_explicit_identities_associated_with_positive_self_similar_Markov_processes} which appears as a scaling limit for the perimeter process. As an application we sharpen the upper bound of the sub-diffusivity exponent for random walk of~\cite{Benjamini-Curien:Simple_random_walk_on_the_uniform_infinite_planar_quadrangulation_subdiffusivity_via_pioneer_points}.
\end{abstract}

\section{Introduction}

One of the main tool to study random maps is the so-called \emph{peeling procedure} which is a step-by-step Markovian exploration of the map. Such a procedure was conceived by Watabiki~\cite{Watabiki:Construction_of_non_critical_string_field_theory_by_transfer_matrix_formalism_in_dynamical_triangulation} and then formalized and used in the setting of the Uniform Infinite Planar Triangulation (UIPT in short) by Angel~\cite{Angel:Growth_and_percolation_on_the_UIPT}. In this paper we will use Budd's peeling process~\cite{Budd:The_peeling_process_of_infinite_Boltzmann_planar_maps} which enables to treat many different models in a unified way. A peeling of a rooted map $\map$ is an increasing sequence of sub-maps with holes
\[\mathfrak{e}_0 \subset \mathfrak{e}_1 \subset \dots \subset \map,\]
where $\mathfrak{e}_0$ is the root-edge of $\map$, and where $\mathfrak{e}_{i+1}$ is obtained from $\mathfrak{e}_i$ by selecting an edge (the edge to peel) on the boundary of one of its holes and then either gluing one face to it, or identifying it to another edge of the same hole; in the latter case this may split the hole into two new holes or make it disappear. When the map $  \mathfrak{m} = \map_{\infty}$ is infinite and one-ended (a map of the plane) we usually ``fill-in'' the finite holes of $\mathfrak{e}_i$ and consider instead a sequence of sub-maps with one hole
\[\e_0 \subset \e_1 \subset \dots \subset \map_\infty\]
called a \emph{filled-in} peeling process. Obviously, the peeling process depends on the algorithm used to choose the next edge to peel and any algorithm is allowed, provided it does not use the information outside $ \e_{n}$ (it is `Markovian') see~\cite{Curien:Peccot} for details. 

In this work we will study the peeling process on random infinite critical Boltzmann planar maps with bounded face-degrees. More precisely, if $\q = (q_i)_{i \ge 1}$ is a non-negative sequence with finite support we define the $\q$-Boltzmann weight of a finite map $\map$ as
\[w(\map) \enskip=\enskip \prod_{f\in \mathsf{Faces}(\map)} q_{\deg(f)}.\]
We suppose that $\q$ is \emph{admissible} (the above measure has finite total mass) and \emph{critical} in the sense that we cannot increase any of the $q_{i}$'s and remain admissible, see~\cite{Budd:The_peeling_process_of_infinite_Boltzmann_planar_maps, Curien:Peccot,Marckert-Miermont:Invariance_principles_for_random_bipartite_planar_maps}. Under these conditions, it was proved by Stephenson~\cite{Stephenson:Local_convergence_of_large_critical_multi_type_Galton_Watson_trees_and_applications_to_random_maps} (see also~\cite{Bjornberg-Stefansson:Recurrence_of_bipartite_planar_maps} in the case of bipartite maps, when all faces have even degrees) that such a Boltzmann map conditioned to be large converges in distribution for the \emph{local topology} to a random infinite map of the plane $\Map$ which is called the $\q$-infinite Boltzmann map. This result encompasses previous ones due to Angel \& Schramm~\cite{Angel-Schramm:UIPT} and Krikun~\cite{Krikun-Local_structure_of_random_quadrangulations} on uniformly random triangulations and quadrangulations, in which case $\Map$ is the UIPT and the UIPQ respectively.

\begin{figure}[!ht]\centering
\includegraphics[scale=.6, page=1]{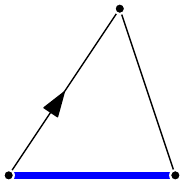}
\includegraphics[scale=.6, page=2]{peeling}
\includegraphics[scale=.6, page=3]{peeling}
\includegraphics[scale=.6, page=4]{peeling}
\includegraphics[scale=.6, page=5]{peeling}
\includegraphics[scale=.6, page=6]{peeling}
\caption{Example of the first few steps of a (filled-in) peeling process; the next edge to peel is indicated in fat blue, and when this edge is identified with another one, the latter is indicated in dashed blue. The filled-in holes are indicated in pink (we do not represent the submap they contain). At the last step, the root-edge gets swallowed. }
\label{fig:peeling}
\end{figure}

Let $( \e_{i} ; i \geq 0)$ be a (filled-in Markovian) peeling of $ \mathfrak{M}_{\infty}$. The proof of~\cite[Theorem 1]{Benjamini-Curien:Simple_random_walk_on_the_uniform_infinite_planar_quadrangulation_subdiffusivity_via_pioneer_points} recast below into Proposition~\ref{prop:hauteur_peeling}  shows that (at least in the case of the UIPQ) such an exploration cannot escape too fast to infinity in the sense that whatever the peeling algorithm used $\overline{\mathfrak{e}}_{n}$ cannot reach distances more than $n^{1/3+o(1)}$ from the origin inside  $ \mathfrak{M}_{\infty}$. On the other hand, it is known (see~\cite[Corollary 27]{Curien:Peccot} based on~\cite{Curien-Le_Gall:Scaling_limits_for_the_peeling_process_on_random_maps}) that the exploration will eventually reveal the whole map, that is
\[\bigcup_{n \ge 0} \e_n = \Map, \quad  \text{a.s.}\]
Our main result gives a quantitative version of the last display. We let $\partial \e_n$ denote the boundary of the sub-map $\e_n$, defined as the edges adjacent to its unique hole.

\begin{thm}[The peeling swallows the root]
\label{thm:peeling_avale_racine}
Let $\Map$ be an infinite Boltzmann planar map with bounded face-degrees and let $(\e_n)_{n \ge 0}$ be any (filled-in Markovian) peeling process of $\Map$. 
If $\root$ denotes the root-edge of $\Map$ then we have as $n \to \infty$:
\[\P(\root \in \partial \e_n) \le n^{-2c/3 + o(1)},\]
where $c \approx 0.1283123514178324542367448657387285493314266204833984375...$ is the positive solution to 
\[\frac{4\pi}{3} = \int_{0}^{1} x^{c-1}(1-x)^{1/2} \d x \cdot \int_{0}^{1/2} x^{c+1/2}(1-x)^{-5/2} \d x.\]
\end{thm}

Of course, a given peeling process may swallow the root-edge much faster, but our upper bound holds for any  peeling algorithm, and it is sharp for the worst algorithm, in which we always select the edge which is at the opposite of the root-edge on $\partial \e_{n}$. Theorem~\ref{thm:peeling_avale_racine} holds with the same proof  for  Angel's peeling process on the UIPT or UIPQ. The intriguing constant $c \approx 0.1283...$  comes from a calculation done on the L\'evy process arising in the Lamperti representation of the $3/2$-stable spectrally negative L\'evy process conditioned to stay positive. Theorem~\ref{thm:peeling_avale_racine} should hold true for any critical generic Boltzmann map. Let us note that the proof of Theorem~\ref{thm:peeling_avale_racine} actually shows that for any sub-map $\e$ with a unique hole and any edge $E \in \partial \e$, if $\ell$ denotes the length of the hole, then if we start the peeling process with $\e_0 = \e$, we have
\[\P(E \in \partial \e_n) \le (n/\ell)^{-2c/3 + o(1)}.\]

\paragraph{Application.}
A particular peeling algorithm was designed in~\cite{Benjamini-Curien:Simple_random_walk_on_the_uniform_infinite_planar_quadrangulation_subdiffusivity_via_pioneer_points} to study the \emph{simple random walk} on the UIPQ and show it is sub-diffusive with exponent $1/3$: the maximum displacement of the walk after $n$ steps is at most of order $n^{1/3+o(1)}$. The proof in~\cite{Benjamini-Curien:Simple_random_walk_on_the_uniform_infinite_planar_quadrangulation_subdiffusivity_via_pioneer_points} actually does not use much of the random walk properties since it is valid for any peeling algorithm (see Proposition~\ref{prop:hauteur_peeling} below). It shows in fact that the maximal displacement of the walk until the discovery of the $n$th \emph{pioneer point} is $n^{1/3+o(1)}$. To improve it we give a bound on the number of pioneer points. For convenience we state the result for the walk on the dual of the UIPQ which is more adapted to Budd's peeling process but the proof could  likely be adapted to the walk on the dual or primal lattice of any generic infinite Boltzmann map.  A \emph{pioneer edge} is roughly speaking a dual edge crossed by the walk, such that its target is not disconnected from infinity by the past trajectory; we refer to Section~\ref{sec:sous_diff} and Figure~\ref{fig:marche_dual} for a formal definition.

\begin{prop}  \label{cor:sous-diff} Let $( \vec{E}_n)_{n \ge 0}$ be the oriented dual edges visited  by a simple random walk on the faces of the UIPQ started from the face adjacent to the right of the root-edge. There exists $\gamma >0$ such that if  $\pi_{n}$ is the number of pioneer edges of the walk up to time $n$ then we have 
\[\Es{\pi_{n}} = O(n^{1- \gamma}).\]
We deduce that
\[\left(n^{-(1-\gamma)/3}\max_{0 \le k \le n} \dgr( \vec{E}_0, \vec{E}_n)\right)_{n \ge 0} \qquad \text{is tight.}\]
\end{prop}

The basic idea is to use reversibility to convert the discovery of a pioneer edge at time $n$ into the event in which the root-edge is still exposed after $n$ steps of the walk as in~\cite[Lemma 12]{Curien:PSHT} and then use Theorem~\ref{thm:peeling_avale_racine}.  Our proof gives a rather low value $\gamma = c/(12+2c) \approx 0.01046881621...$ with $c$ from Theorem~\ref{thm:peeling_avale_racine}; it could optimised, but it cannot exceed 
$\gamma 
= 2c/(3+2c)
\approx 0.07880082179...$ 
and this would require a longer argument to yield a sub-diffusivity exponent slightly larger than $0.3$, which is still far from the conjectural value of $1/4$ for the simple random walk on generic random planar maps. We therefore restricted ourselves to this weaker exponent which still improves on the very general upper bound on sub-diffusivity exponent for unimodular random graphs discovered by Lee~\cite[Theorem 1.9]{Lee:Conformal_growth_rates_and_spectral_geometry_on_distributional_limits_of_graphs}. In the case of the UIPT (type II) Gwynne \& Miller~\cite[Theorem 1.8]{Gwynne-Miller:Random_walk_on_random_planar_maps_spectral_dimension_resistance_and_displacement} have proven that the exponent is at least $1/4$ and an ongoing work of Gwynne \& Hutchcroft~\cite{Gwynne-Hutchcroft:preparation} shall yield the corresponding upper bound (still in the case of the UIPT). We do not think that our methods can lead to this optimal exponent.

The results of this paper can also be adapted to the case of peeling processes (and random walks on the dual and primal) on Boltzmann maps with large faces~\cite{Curien-Marzouk:sous_diff}.

\paragraph{Acknowledgement.}
We thank the anonymous referee for a swift report. We acknowledge the support of the grants \texttt{ANR-15-CE40-0013} `Liouville' and \texttt{ANR-14-CE25-0014} `GRAAL', the grant ERC `GeoBrown', the Institut Universitaire de France and of the Fondation Math\'ematique Jacques Hadamard.

\section{Peeling estimates}

The proof of Theorem~\ref{thm:peeling_avale_racine} relies on the \emph{perimeter process} associated with a peeling process. We first briefly recall from~\cite{Budd:The_peeling_process_of_infinite_Boltzmann_planar_maps} the law of this process,  a random walk conditioned to stay positive which converges once suitably rescaled towards a $3/2$-stable spectrally negative Lévy process conditioned to stay positive; the quantities that we will consider do not depend on scaling constants so we shall not need to know the precise characteristic exponent of this Lévy process and we shall refer to it as \emph{the} $3/2$-stable spectrally negative Lévy process. Throughout this section, $\Map$ denotes a random $\q$-infinite Boltzmann map, where $\q$ is an admissible and critical weight sequence with finite support.

\subsection{Peeling of a Boltzmann planar map}
\label{sec:peeling}

Let us give a brief description of the peeling process of Budd~\cite{Budd:The_peeling_process_of_infinite_Boltzmann_planar_maps}, with a point of view inspired by~\cite{Budd-Curien:Geometry_of_infinite_planar_maps_with_high_degrees,Curien:Peccot} to which we refer for more details. Fix an infinite one-ended map $\map_\infty$; a \emph{sub-map} $\e$ of $\map_\infty$ is a map with a distinguished simple face called the \emph{hole}, such that one can recover $\map_\infty$ by gluing inside the unique hole of $\e$ a planar map with a (not necessarily simple) boundary of perimeter matching that of the hole of $\e$ (this map is then uniquely defined). Then a \emph{(filled-in) peeling exploration} of $\map_\infty$ is an increasing sequence $(\e_i)_{i \ge 0}$ of sub-maps of $\map_\infty$ containing the root-edge. Such an exploration depends on an algorithm $\Alg$ which associates with each sub-map $\e$ an edge on the boundary of its hole; this algorithm can be deterministic or random, but in the latter case the randomness involved must be independent of the rest of the map. Then the peeling exploration of $\map_\infty$ with algorithm $\Alg$ is the sequence $(\e_i)_{i \ge 0}$ of sub-maps of $\map_\infty$ constructed recursively as follows. First $\e_0$ always only consists of two simple faces of degree $2$ and an oriented edge, the hole is the face on the left of this root-edge; this corresponds to the root-edge of the map $\map_\infty$ that we open up in two. Then for each $i \ge 0$, conditional on $\e_i$, we choose the edge $\Alg(\e_i)$ on the boundary of its hole, and we face two possibilities depicted in Figure~\ref{fig:peeling_dual}:
\begin{enumerate}
\item\label{item:def_peeling_new_face} either the face in $\map_\infty$ on the other side of $\Alg(\e_i)$ is not already present in $\e_i$; in this case $\e_{i+1}$ is obtained by adding this face to $\e_i$ glued onto $\Alg(\e_i)$ and without performing any other identification of edges,
\item\label{item:def_peeling_gluing} or the other side of $\Alg(\e_i)$ in $\map_\infty$ actually corresponds to a face already discovered in $\e_i$. In this case $\e_{i+1}$ is obtained by performing the identification of the two edges in the hole of $\e_i$. This usually creates two holes, but since $\map_\infty$ is one-end, we decide to fill-in the one containing a finite part of $\map_\infty$. 
\end{enumerate}

\begin{figure}[!ht]
\begin{center}
\includegraphics[height=12\baselineskip]{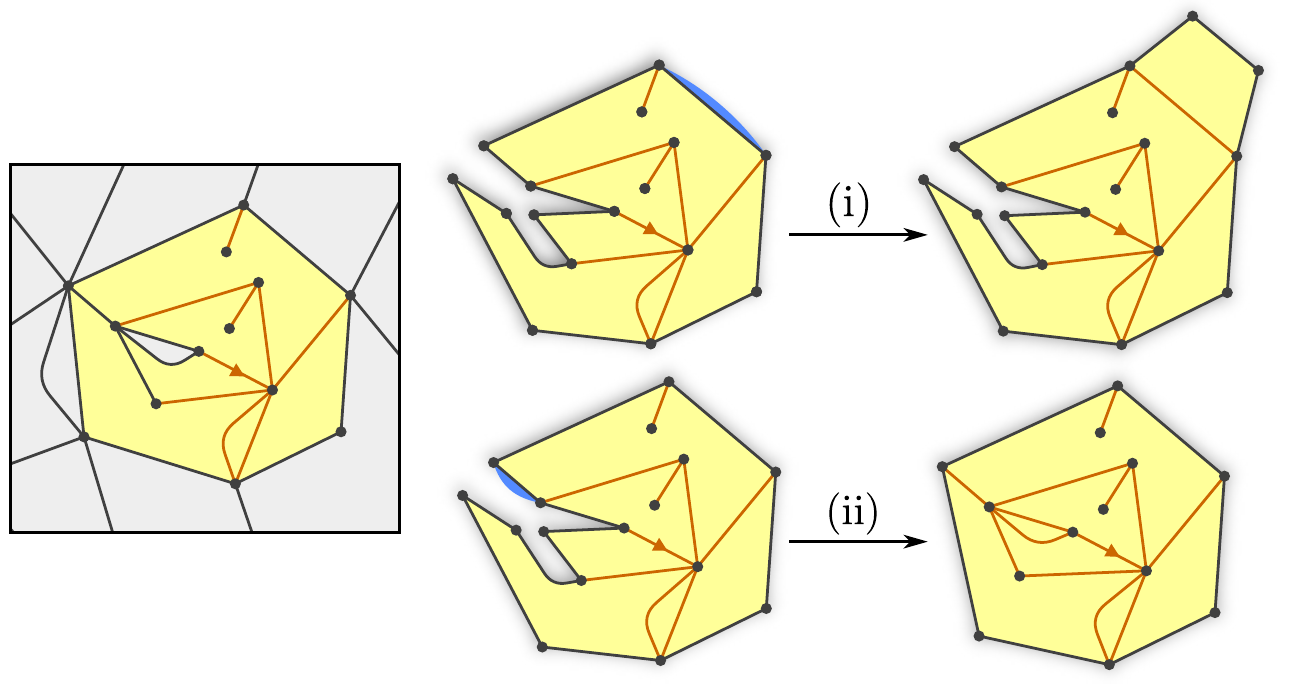}
\caption{Illustration of a (filled-in) peeling step in a one-ended bipartite map. The peeling edge is depicted in blue. In the first case we add a new face adjacent to this edge, and in the second case we identify two edges on the boundary of the hole, this splits the hole into two components and we fill-in the finite one.
\label{fig:peeling_dual}}
\end{center}
\end{figure}

We define the perimeter of a sub-map $\e$ of $\Map$ to be the number of edges on the boundary of its simple hole, we denote it by $|\partial \e|$. If $(\e_n)_{n \ge 0}$ is a peeling exploration of $\Map$, we consider its \emph{perimeter process} given by $(|\partial \e_n|)_{n \ge 0}$; note that $|\partial \e_0| = 2$ from our convention. Budd~\cite{Budd:The_peeling_process_of_infinite_Boltzmann_planar_maps} defines in terms of $\q$ a probability measure $\nu$; we shall denote by $S = (S_n)_{n \ge 0}$ a random walk started from $2$ and with step distribution $\nu$. When $\q$ is admissible, critical and finitely supported, the support of the law $\nu$ is bounded above, furthermore this law is centered and belongs to the domain of attraction of a stable law with index $3/2$; precisely,
\begin{equation}\label{eq:queue_loi_nu}
\nu(-k) \sim C_\nu \cdot k^{-5/2}
\quad\text{as}\quad
k \to \infty, \text{ for some constant }C_\nu > 0.
\end{equation}

\begin{lem}[Budd~\cite{Budd:The_peeling_process_of_infinite_Boltzmann_planar_maps}]\label{lem:budd15}
Let $(\e_n)_{n \ge 0}$ be any peeling exploration of $\Map$. The perimeter process $(|\partial \e_n|)_{n \ge 0}$ is a Markov chain whose law does not depend on the peeling algorithm $\Alg$. More precisely, it has the same law as $S^\uparrow$, a random walk with step distribution $\nu$ conditioned to stay positive (see e.g.~\cite{Bertoin-Doney:On_conditioning_a_random_walk_to_stay_nonnegative}).
Moreover, the convergence in distribution
\[(n^{-2/3} |\partial \e_{[nt]}|)_{t \ge 0}
\cvdist
(\Upsilon^\uparrow_t)_{t \ge 0},\]
holds in the sense of Skorokhod, where $\Upsilon^\uparrow$ is a version conditioned to never hit $(-\infty, 0)$ of a $3/2$-stable spectrally negative Lévy process (see e.g.~\cite{Chaumont:Conditionings_and_path_decompositions_for_Levy_processes}).
\end{lem}

Let us  mention that a jump of this process, say $|\partial \e_n|-|\partial \e_{n-1}| = k$ with $k \in \Z$, corresponds in the $n$-th step of the peeling exploration to the discovery of a new face of degree $k+2$ if $k \ge -1$, that is case~\ref{item:def_peeling_new_face} in Figure~\ref{fig:peeling_dual}, and it corresponds to case~\ref{item:def_peeling_gluing} in Figure~\ref{fig:peeling_dual} if $k \le -2$: the selected edge $\Alg(\e_n)$ is identified to another one on the boundary, and `swallows a bubble' of length $-(k+2)$, either to the left or to the right, equally likely.

 \subsection{Proof of Theorem~\ref{thm:peeling_avale_racine}}

We now prove Theorem~\ref{thm:peeling_avale_racine}, that is: for any peeling process, we have $\P(\root \in \partial \e_n) \le n^{-2c/3+o(1)}$ and we will compute the value $c$ in the next section. We first consider the case of the UIPT which corresponds to $\Map$ with $ \mathbf{q} = (432^{-1/4} \mathbf{1}_{k=3})_{k \geq 1}$, for which $\nu$ is supported by $\{\dots, -2,-1,0,1\}$.

\begin{proof}[Proof of Theorem~\ref{thm:peeling_avale_racine} for the UIPT]
First observe that on the event $\{\root \in \partial \e_n\}$, for each $0 \le k \le n-1$, if we identify the peel edge $\Alg(\e_k)$ to another one on the boundary and swallow a bubble longer than $\frac{1}{2} |\partial \e_k|$, then this identification must be on the correct side, so that the root-edge does not belong to this bubble. If we put $P_k = |\partial \e_k|$ and $\Delta P_k = P_{k+1}-P_k$, then it follows from the Markov property that
\[\P(\root \in \partial \e_n) \le \Es{\prod_{k=0}^{n-1} 2^{-1} \ind{\Delta P_k < - P_k/2}}
\le \Es{2^{- \#\{0 \le k \le n-1 : \Delta P_k < - P_k/2\}}}.\]
We first claim that appart from an event whose probability is $op(n)$ i.e.~decays faster than any polynomial, the process $P$ reaches values of order $n^{2/3+o(1)}$ in the first $n$ steps: From the convergence of the process $(P_n)_{n \ge 0}$ in Lemma~\ref{lem:budd15}, we see that $P_n$ is of order $n^{2/3}$; in particular, there exists $\eta \in (0,1)$ such that for every $N$ large enough, for every integer $z \in (0, N^{2/3})$, we have $\P_z(\sup_{0 \le k \le N} P_k > N^{2/3}) > \eta$. Then fix $\varepsilon > 0$; by splitting the interval $[0,n]$ in $n^{3\varepsilon/2}$ sub-intervals of length $n^{1-3\varepsilon/2}$, we deduce from the Markov property that for every $n$ large enough,
\begin{align*}
\Pr{\sup_{0 \le k \le n} P_k \le n^{2/3 - \varepsilon}}
&\le \left(\sup_{1 \le z \le n^{2/3 - \varepsilon}} \P_z\left(\sup_{0 \le k \le n^{1-3\varepsilon/2}} P_k \le n^{2/3 - \varepsilon}\right)\right)^{n^{3\varepsilon/2}}
\\
&\le (1-\eta)^{n^{3\varepsilon/2}}
\\
&= op(n).
\end{align*}
We then decompose the process into scales: for every $ i \ge 0$, let us put $\tau_i = \inf\{k \ge 0 : P_k \ge 2^{i}\}$; observe that since the jumps of $P$ are bounded by one, then $P_{\tau_i} = 2^{i}$. Using the last display and the strong Markov property we deduce that 
\[\Es{2^{- \#\{0 \le k \le n-1 : \Delta P_k < - P_k/2\}}}
\le \prod_{i=0}^{\log_2(n^{2/3 - \varepsilon})} \E_{2^{i}}\left[2^{- \#\{0 \le k < \tau_{i+1} : \Delta P_k < - P_k/2\}}\right] + op(n).\]
If $\theta(z) = \inf\{t \ge 0 : \Upsilon^\uparrow_t \ge z\}$ is the first passage time above level $z > 0$, the convergence of the process $(P_n)_{n \ge 0}$ in Lemma~\ref{lem:budd15} then implies that 
\begin{eqnarray}\E_{2^{i}}\left[2^{- \#\{0 \le k < \tau_{i+1} : \Delta P_k < - P_k/2\}}\right] &\cv[i]&
\E_1\left[2^{- \#\{0 \le t \le \theta(2) : \Delta \Upsilon^\uparrow_t < - \Upsilon^\uparrow_{t-}/2\}}\right] \nonumber \\ & \eqqcolon& 
 \ex^{-c \ln 2},  \label{eq:defc} \end{eqnarray} where the last line defines the constant $c$. Putting back the pieces together, we deduce by Ces\`{a}ro summation that for $n$ large enough we have
\[\P(\root \in \partial \e_n)
\le \ex^{-c \ln(2) \log_2(n^{2/3 - \varepsilon}) + o(1)} = n^{-2c/3 + o(1)}\]
and the proof in the case of the UIPT is complete.
\end{proof}

For more general maps, the perimeter process does not increase only by one, so it is not true that $|\e_{\tau_i}| = 2^{i}$ a.s.~and the scales are not independent. Nonetheless, if the degrees are uniformly bounded, say by $D < \infty$, then $|\e_{\tau_i}| \in \{2^{i}, \dots, 2^{i}+D\}$ and the scales do become independent at the limit: one can still use the strong Markov property to get the bound
\[\Es{2^{- \#\{0 \le k \le n-1 : \Delta P_k < - P_k/2\}}}
\le \prod_{i=1}^{\log_2(n^{2/3 - \varepsilon})} \sup_{2^{i} \le z\le 2^{i}+D} \E_z\left[2^{- \#\{0 \le k < \tau_{i+1} : \Delta P_k < - P_k/2\}}\right],\]
and $\sup_{2^{i} \le z\le 2^{i}+D} \E_z[2^{- \#\{0 \le k < \tau_{i+1} : \Delta P_k < - P_k/2\}}]$ converges as previously to $ \mathrm{e}^{-c \ln 2}$.

\medskip
\noindent \textbf{Remark:} The factor $2/3$ coming from the lower bound $\sup_{0 \le k \le n} P_k \ge n^{2/3 - \varepsilon}$ with high probability cannot be improved. Indeed it is easy to check in our case of increments bounded above that $\P(\sup_{0 \le k \le n} P_k \le n^{2/3 + \varepsilon})=op(n)$. This shows that the exponent in Theorem~\ref{thm:peeling_avale_racine} is optimal for the worst algorithm which always peels at the opposite of the root-edge on $\partial \e_n$.

\subsection{\texorpdfstring%
	{Calculation of $c$ via Lamperti representation}%
	{Calculation of c via Lamperti representation}}

We now characterize the value $c$ defined in~\eqref{eq:defc} using the Lamperti representation of $\Upsilon^\uparrow$. We start with a simple calculation on general L\'evy processes. \medskip

Let $\xi=(\xi_{t})_{ t \geq 0}$ be a L\'evy processes with drift $a \in \R$, no Brownian part, and L\'evy measure $\Pi(\d y)$ which we assume possesses no atom and is supported on $\R_-$ so $\xi$ makes only negative jumps (see e.g. Bertoin's book~\cite[Chapter VII]{Bertoin:Levy_processes}); assume also that $\xi$ does not drift towards $-\infty$. We can consider the Laplace transform of $\xi$ and define its characteristic exponent via the L\'evy--Khintchine representation as 
\[\Es{\ex^{\lambda \xi_t}} = \ex^{t\psi(\lambda)},
\quad\text{for all }
t, \lambda \geq 0,
\enskip\text{where} \enskip
\psi(\lambda) = a \lambda + \int_{\R_-} \big( \ex^{\lambda x}-1- \lambda x \ind{|x|<1}\big) \Pi(\d x).\]
Recall that $\theta(z) = \inf\{t \ge 0 : \xi_t \ge z\}$ is the first passage time above level $z > 0$; since $\xi$ does not drift towards $-\infty$ and does not make positive jumps, we have $\theta(z) < \infty$ and $\xi_{\theta(z)}=z$ almost surely.

\begin{lem}\label{lem:controle_sauts_Levy} Let $F : \R_- \to \R_-$ be a function which is identically $0$ in a neighborhood of $0$. If we note $\Delta \xi_{t} = \xi_t - \xi_{t-} \le 0$ for the value of the jump at time $t$, then for every $z \geq 0$ we have 
\[\Es{\exp\left( \sum_{ t \leq \theta(z) } F( \Delta \xi_{t})\right)} = \exp(- c_{F} \cdot z),\]
where $c_{F}>0$ satisfies
\[\psi(c_{F}) =  \int_{\R_-} \left( 1-\ex^{F(x)}\right) \ex^{c_{F} x} \Pi(\d x).\]
\end{lem}

\begin{proof}
For $c>0$ we consider the positive càdlàg process 
\[Z_{c}(t)=\exp\left(c \xi_{t} + \sum_{s \le t} F( \Delta \xi_s)\right).\]
Clearly $\log Z_{c}(t)$ has stationary and independent increments. We will choose $c$ so that $Z_{c}$ is a martingale and for this it is sufficient to tune $c$ so that $ \E[Z_{c}(t)]=1$. Appealing to the Lévy--It\=o decomposition and the exponential formula for Poisson random measure (see e.g.~\cite{Bertoin:Levy_processes}, Chapter 0.5 and Chapter 1, Theorem 1 and its proof), and using the L\'evy--Khintchine representation of $\psi(\cdot)$, we get
\begin{align*}
\Es{\exp\left(c \xi_{t} + \sum_{s \le t} F( \Delta \xi_s)\right)}
&= \exp\left(t\left(ac + \int_{\R_-} (\ex^{c x + F(x)} - 1 - cx \ind{|x| < 1}) \Pi(\d x)\right)\right)
\\
&= \exp\left(t\left(\psi(c) + \int_{\R_-} (\ex^{F(x)} - 1) \ex^{c x} \Pi(\d x)\right)\right).
\end{align*}
Hence if we pick $c=c_{F}$ satisfying the assumption of the lemma then $Z_{c}$ is a positive martingale. Furthermore by our assumptions $Z_{c}(t \wedge \theta(z))$ is bounded by $\ex^{cz}$ so we can apply the optional sampling theorem and get the statement of the lemma (using also that $\xi_{\theta(z)} = z$).
\end{proof}

Let us apply this result to a well-chosen Lévy process. Let $(\Upsilon^\uparrow_{t})_{t \geq 0}$ be the $3/2$-stable Lévy process with no negative jumps conditioned to stay positive and started from $\Upsilon^\uparrow_{0}=1$. The constant $c$ defined in~\eqref{eq:defc} satisfies  $\ex^{- c \ln 2}= \E[2^{- \mathcal{N}}]$ where 
\[\mathcal{N} = \# \left\{\text{jumps before } \theta(2), \text{such that } |\Delta \Upsilon^\uparrow_{t}| > \frac{\Upsilon^\uparrow_{t-}}{2}\right\}.\]
Note that, as mentioned in the introduction of this section, this quantity is scale-invariant so it does not depend on the choice of the normalization of $\Upsilon^\uparrow$. The process $\Upsilon^\uparrow$ is a so-called \emph{positive self-similar Markov process}, and by the Lamperti representation it can be represented as the (time-changed) exponential of a L\'evy process: for every $t \ge 0$,
\[\Upsilon^\uparrow_{t} = \exp(\xi_{\tau_{t}}),\]
where the random time change $\tau_{t}$ will not be relevant in what follows. In particular, through this representation, for every $u \in(0,1)$, the random variable $u^{\mathcal{N}}$ is transformed into the variable
\[\exp\left( \sum_{t \leq \theta (\ln 2)} \ln u \cdot \ind{ \Delta \xi_{t} < - \ln 2}\right)\]
for the associated L\'evy process $\xi$. This is of the form of the previous lemma with $F : x \mapsto \ln u \cdot \ind{x < - \ln 2}$. Caballero and Chaumont~\cite[Corollary 2]{Caballero-Chaumont:Conditioned_stable_Levy_processes_and_the_Lamperti_representation} have computed explicitly the characteristics of $\xi$; in particular $\xi$ drifts towards $+\infty$ and has a L\'evy measure given by 
\[\Pi(\d y)= \frac{\ex^{3y/2}}{(1- \ex^{y})^{5/2}} \ind{y <0} \d y.\]
Furthermore in our spectrally negative case, its Laplace exponent has been computed in Chaumont, Kyprianou and Pardo~\cite[just before Lemma 2]{Chaumont-Kyprianou-Pardo:Some_explicit_identities_associated_with_positive_self_similar_Markov_processes}: for every $\lambda \ge 0$,
\[\psi(\lambda) = \E[\xi_{1}] \frac{\Gamma(\lambda + 3/2)}{\Gamma(\lambda) \Gamma(3/2)}.\]
To compute the value of $\E[\xi_{1}]$ we can use the asymptotic behaviour as $\lambda \to \infty$:
\[\frac{\E[\xi_{1}]}{\Gamma(3/2)} \lambda^{3/2}
\underset{\lambda \to \infty}{\sim}
\E[\xi_{1}] \frac{\Gamma(\lambda + \frac{3}{2})}{\Gamma(3/2) \Gamma(\lambda)}
= \psi(\lambda)
= a \lambda + \int_{-\infty}^{0}  \frac{\d y \ex^{3y/2}}{(1- \ex^{y})^{5/2}}\left(\ex^{\lambda y}-1-\lambda y \ind{|y|<1}\right).\]
 It is easy to see that in the right-hand side of the last display, when $\lambda \to \infty$, the only contribution which can be of the order of $\lambda ^{3/2}$ must appear in the vicinity of $0$, hence we can replace the L\'evy measure by its equivalent as $y \to 0$ and the last display becomes
\[\frac{\E[\xi_{1}]}{\Gamma(3/2)} \lambda^{3/2}
\underset{\lambda \to \infty}{\sim}
\int_{-\infty}^{0}  \frac{\d y}{|y|^{5/2}}\left( \ex^{\lambda y}-1-\lambda y \ind{|y|<1}\right)
\underset{\lambda \to \infty}{\sim}
\frac{4\sqrt{\pi}}{3}\lambda^{3/2},\]
where the last asymptotical equivalence is a standard calculation. Hence we deduce that 
$\E[\xi_{1}] = \frac{2\pi}{3}.$ We can thus apply the last lemma and conclude that for every $u \in (0,1)$, 
\[\Es{u^{\mathcal{N}}} = \Es{\exp\left( \sum_{t \leq \theta (\ln 2)} \ln u \cdot \ind{ \Delta \xi_{t} < - \ln 2}\right)} = \exp(- c_u \ln 2),\]
where $c_u$ satisfies
\[\frac{2\pi}{3} \frac{\Gamma(c_u + 3/2)}{\Gamma(c_u) \Gamma(3/2)}
= \psi(c_u)
= \int_{-\infty}^{\ln u} (1-u) \ex^{c_u y} \frac{\ex^{3y/2}}{(1- \ex^{y})^{5/2}} \d y.\]
Performing the change of variable $x = \ex^{y}$ in the second integral and using the representation $\frac{\Gamma(a)\Gamma(b)}{\Gamma(a+b)} = \int_{0}^{1} x^{a-1}(1-x)^{b-1} \d x$ gives us the equivalent definition of $c_u$ by an integral equation:
\[\frac{2\pi}{3(1-u)}
= \int_{0}^{1} x^{c_u-1}(1-x)^{1/2} \d x \cdot \int_0^u x^{c_u+1/2} (1- x)^{-5/2} \d x.\]
The constant $c$ defined in~\eqref{eq:defc} and which appears in Theorem~\ref{thm:peeling_avale_racine} corresponds to $u = 1/2$, and it can easily be evaluated numerically to find that $c = c_{1/2} \approx 0.1283123514178324542...$

\section{Pioneer edges and sub-diffusivity}
\label{sec:sous_diff}

In this section we suppose that $ \mathbf{q} = ( \frac{1}{12} \mathbf{1}_{k=4})_{k \geq 1}$ so that $ \mathfrak{M}_{\infty}$ is the UIPQ although that there is little doubt that everything can be extended to the more general context of  generic Boltzmann maps. We first recast the main result of \cite{Benjamini-Curien:Simple_random_walk_on_the_uniform_infinite_planar_quadrangulation_subdiffusivity_via_pioneer_points} in the case of Budd's peeling process. We denote by $\rho$ the origin of the root-edge. 

\begin{prop}[Depth of a peeling process]
\label{prop:hauteur_peeling}
Let  $(\e_n)_{n \ge 0}$ be any (filled-in Markovian) peeling process of the UIPQ with origin $\rho$, then the sequence
\[\left(n^{-1/3} \max\{\dgr( \rho, x) ; x \in \e_n\}\right)_{n \ge 1} \qquad \mbox{ is tight}.\]
\end{prop}

\begin{proof}[Sketch of proof.] We  follow~\cite[Section 4.2]{Benjamini-Curien:Simple_random_walk_on_the_uniform_infinite_planar_quadrangulation_subdiffusivity_via_pioneer_points}. Throughout the proof, if $(Y_n)_{n \ge 0}$ and $(Z_n)_{n \ge 0}$ are two positive processes, we shall write $Y_n \lesssim Z_n$ if $(Y_{n}/Z_{n})_{n \geq 0}$ is tight. 
Also $Y_n \gtrsim Z_n$ if $Z_n \lesssim Y_n$, and finally $Y_n \asymp Z_n$ if we have both $Y_n \lesssim Z_n$ and $Y_n \gtrsim Z_n$. \\
For every $n \ge 0$, let $D_n^-$ and $D_n^+$ be respectively the minimal and the maximal distance to the origin $\rho$ of a vertex in $\partial \e_n$ so that   \begin{eqnarray} \label{eq:inclus} \hBall(\Map, D_n^- -1) \subset \e_n \subset \hBall(\Map, D_n^+ +1),  \end{eqnarray} where $\hBall(\Map, r)$ denotes the hull of the ball in $\Map$ of radius $r$ centered at the origin of $\Map$. Using $|\e_n| \asymp n^{4/3}$  (see~\cite{Budd:The_peeling_process_of_infinite_Boltzmann_planar_maps}) and $|\hBall(\Map, r)| \asymp r^4$ (see ~\cite{Curien-Le_Gall:Scaling_limits_for_the_peeling_process_on_random_maps}) the two inclusions of~\eqref{eq:inclus} give $D_{n}^{-} \lesssim  n^{1/3}$ and $D_{n}^{+} \gtrsim n^{1/3}$.

We will first show that $D_n^+ \lesssim n^{1/3}$ or simply $D_{n}^{+}-D_{n}^{-} \lesssim n^{1/3}$ since we already know $D_{n}^{-} \lesssim n^{1/3}$. But as in~\cite[Section 4.2]{Benjamini-Curien:Simple_random_walk_on_the_uniform_infinite_planar_quadrangulation_subdiffusivity_via_pioneer_points} we can bound $D_{n}^{+}-D_{n}^{-}$ by the aperture of $ \mathfrak{M}_{\infty} \backslash \overline{e}_{n}$ which is the maximal graph distance between points on its (general) boundary. Since conditionally on $ \overline{\mathfrak{e}}_{n}$ the later is a UIPQ with a boundary of perimeter $| \partial \overline{ \mathfrak{e}}_{n}|$ we can use~\cite[Section 3.2]{Curien-Miermont:Uniform_infinite_planar_quadrangulations_with_a_boundary}  to deduce that its aperture is $\lesssim  \sqrt{| \partial \overline{ \mathfrak{e}}_{n}|}$. Finally since $ | \partial \overline{ \mathfrak{e}}_{n}| \asymp n^{2/3}$ by~\cite{Budd:The_peeling_process_of_infinite_Boltzmann_planar_maps} we indeed deduce that $D_{n}^{+}-D_{n}^{-} \lesssim n^{1/3}$. The end of the proof is then the same as~\cite[Section 4.2]{Benjamini-Curien:Simple_random_walk_on_the_uniform_infinite_planar_quadrangulation_subdiffusivity_via_pioneer_points} and use the fact that the hull of the ball of radius $r$ in the UIPQ has no long tentacles of length $\gg r$. \end{proof}

To prove Proposition~\ref{cor:sous-diff} we shall use a particular peeling algorithm, introduced in~\cite{Benjamini-Curien:Simple_random_walk_on_the_uniform_infinite_planar_quadrangulation_subdiffusivity_via_pioneer_points} that we adapt to the peeling of~\cite{Budd:The_peeling_process_of_infinite_Boltzmann_planar_maps} and the walk on the faces, see also~\cite[Chapter~8]{Curien:Peccot}. We let our walk start from the root-face on the right of the root-edge of $\Map$. Every time the walk wants to cross an edge on the boundary of the explored region, we peel this edge. We then define the pioneer edges of the walk as the peel edges. We let $\pi_{n}$ denote the number of pioneer edges amongst the first $n$ steps of the walk.

\begin{figure}[!ht]\centering
\includegraphics[height=10\baselineskip]{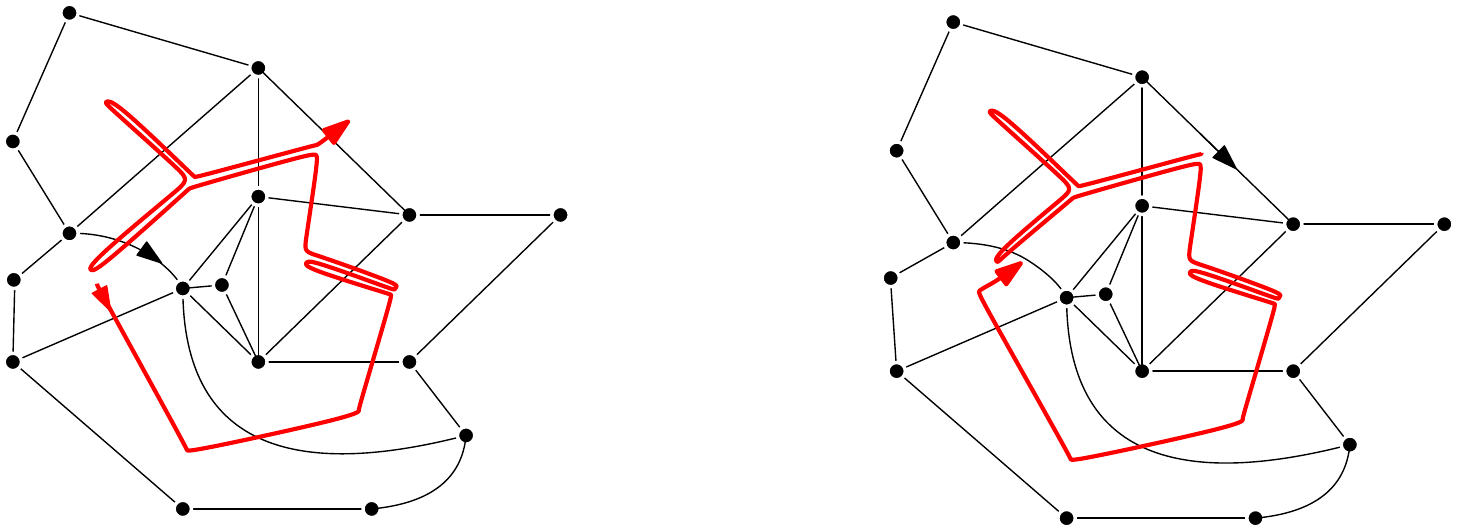}
\caption{Left: A simple random walk on the dual map, started at the face on the right-hand side of the root-edge, and about to reach a pioneer edge. Right: the same scenery in reverse time where the root-edge is not swallowed after $n$ steps of walk.}
\label{fig:marche_dual}
\end{figure}

\begin{proof}[Proof of Proposition~\ref{cor:sous-diff}] We use the reversibility trick~\cite[Lemma 12]{Curien:PSHT}. 
Let $(\vec{E}_i)_{i \ge 1}$ be the edges visited by the walk. Fix $n \geq 1$, since the UIPQ is a stationary and reversible random graph (see~\cite{Benjamin-Curien:Ergodic_theory_on_stationary_random_graphs}) we can reverse the first $n$ steps of the path and deduce that the probability that $\vec{E}_n$ is pioneer equals the probability that (one side of) the root-edge $\root$ of $\Map$ is still on the boundary of $\e_{\pi_n}$ after performing $n$ steps of random walk, see Fig.~\ref{fig:marche_dual}. We then split this probability as follows:
\[\P(\vec{E}_n \text{ is pioneer})
= \Pr{\root \in \partial \e_{\pi_n}}
\le \Pr{\pi_n \leq n^{\alpha}} + \Pr{\root \in \partial \e_{n^\alpha}},\] for some small $\alpha >0$ that we tune later on.
According to Theorem~\ref{thm:peeling_avale_racine}, the second term is bounded by $n^{-2 c \alpha/3 +o(1)}$. Concerning the first term, if $\pi_{n} \leq n^{\alpha}$ then the walk has been confined in $\hBall(\Map, 4 n^{\alpha})$ for the first $n$ steps. However we claim that the probability that the walk stays confined in a given finite region $ G$ for a long time is bounded as follows: 
\[\Pr{X_i \in G \text{ for all } 1 \le i \le n} \le \exp( - n C / |G|^2),\] for some constant $C>0$. 
This can be shown using the fact that the expected cover time of any finite planar graph $G$ is less than $6 |G|^2$, see~\cite{Jonasson-Schramm:On_the_cover_time_of_planar_graphs} which also discusses lower bounds. In our context we thus deduce that
\begin{align*}
\P(\vec{E}_{n} \mbox{ is pioneer}) &\le n^{- 2\alpha c/3 + o(1)} + \P(|\hBall(\Map, 4n^{\alpha})| \geq n^{1/2 + o(1)}) + op(n)
\\
&\le n^{- 2\alpha c/3 + o(1)} + n^{-1/2 + o(1)} \cdot \E[|\hBall(\Map, 4n^{\alpha})|]
\\
&\le n^{- 2\alpha c/3 + o(1)} +  n^{4 \alpha -1/2 +o(1)},
\end{align*}
where the second inequality is an application of the Markov inequality, and the third one may be found in~\cite[Proposition 14]{Le_Gall-Lehericy:Separating_cycles_and_isoperimetric_inequalities_in_the_UIPQ}. 
We can then pick $\alpha = (8+4c/3)^{-1}$ so that the last display is  smaller than $n^{-\gamma+o(1)}$ for $\gamma = c/(12+2c)$. By summing over $n$ this implies the desired estimate $\E[\pi_{n}] = O( n^{1-\gamma})$. The end of the proof it then straightforward. Let us perform $n$ steps of random walk on the dual of the UIPQ. By the above and Markov inequality we deduce that $ \pi_{n} \lesssim n^{1-\gamma}$, and so we have performed $\lesssim n^{1-\gamma}$ peeling steps. By Proposition~\ref{prop:hauteur_peeling} we have thus remained within distance $ \lesssim n^{ (1-\gamma)/3}$ from the origin of the map. \end{proof}


{\small
\linespread{1}\selectfont 

\begin{thebibliography}{10}

\bibitem{Angel:Growth_and_percolation_on_the_UIPT}
{\sc Angel, O.}
\newblock Growth and percolation on the uniform infinite planar triangulation.
\newblock {\em Geom. Funct. Anal. 13}, 5 (2003), 935--974.

\bibitem{Angel-Schramm:UIPT}
{\sc Angel, O., and Schramm, O.}
\newblock Uniform infinite planar triangulations.
\newblock {\em Comm. Math. Phys. 241}, 2-3 (2003), 191--213.

\bibitem{Benjamin-Curien:Ergodic_theory_on_stationary_random_graphs}
{\sc Benjamini, I., and Curien, N.}
\newblock Ergodic theory on stationary random graphs.
\newblock {\em Electron. J. Probab. 17\/} (2012), no. 93, 20.

\bibitem{Benjamini-Curien:Simple_random_walk_on_the_uniform_infinite_planar_quadrangulation_subdiffusivity_via_pioneer_points}
{\sc Benjamini, I., and Curien, N.}
\newblock Simple random walk on the uniform infinite planar quadrangulation:
  subdiffusivity via pioneer points.
\newblock {\em Geom. Funct. Anal. 23}, 2 (2013), 501--531.

\bibitem{Bertoin:Levy_processes}
{\sc Bertoin, J.}
\newblock {\em L\'evy processes}, vol.~121 of {\em Cambridge Tracts in
  Mathematics}.
\newblock Cambridge University Press, Cambridge, 1996.

\bibitem{Bertoin-Doney:On_conditioning_a_random_walk_to_stay_nonnegative}
{\sc Bertoin, J., and Doney, R.~A.}
\newblock On conditioning a random walk to stay nonnegative.
\newblock {\em Ann. Probab. 22}, 4 (1994), 2152--2167.

\bibitem{Bjornberg-Stefansson:Recurrence_of_bipartite_planar_maps}
{\sc Bj{\"o}rnberg, J., and Stef{\'a}nsson, S.~{\"O}.}
\newblock Recurrence of bipartite planar maps.
\newblock {\em Electron. J. Probab. 19\/} (2014), no. 31, 40.

\bibitem{Budd:The_peeling_process_of_infinite_Boltzmann_planar_maps}
{\sc Budd, T.}
\newblock The peeling process of infinite {B}oltzmann planar maps.
\newblock {\em Electron. J. Combin. 23}, 1 (2016), Paper 1.28, 37.

\bibitem{Budd-Curien:Geometry_of_infinite_planar_maps_with_high_degrees}
{\sc Budd, T., and Curien, N.}
\newblock Geometry of infinite planar maps with high degrees.
\newblock {\em Electron. J. Probab. 22\/} (2017), Paper No. 35, 37.

\bibitem{Caballero-Chaumont:Conditioned_stable_Levy_processes_and_the_Lamperti_representation}
{\sc Caballero, M.~E., and Chaumont, L.}
\newblock Conditioned stable {L}\'evy processes and the {L}amperti
  representation.
\newblock {\em J. Appl. Probab. 43}, 4 (2006), 967--983.

\bibitem{Chaumont:Conditionings_and_path_decompositions_for_Levy_processes}
{\sc Chaumont, L.}
\newblock Conditionings and path decompositions for {L}\'evy processes.
\newblock {\em Stochastic Process. Appl. 64}, 1 (1996), 39--54.

\bibitem{Chaumont-Kyprianou-Pardo:Some_explicit_identities_associated_with_positive_self_similar_Markov_processes}
{\sc Chaumont, L., Kyprianou, A.~E., and Pardo, J.~C.}
\newblock Some explicit identities associated with positive self-similar
  {M}arkov processes.
\newblock {\em Stochastic Process. Appl. 119}, 3 (2009), 980--1000.

\bibitem{Curien:PSHT}
{\sc Curien, N.}
\newblock Planar stochastic hyperbolic triangulations.
\newblock {\em Probab. Theory Related Fields 165}, 3-4 (2016), 509--540.

\bibitem{Curien:Peccot}
{\sc Curien, N.}
\newblock Peeling random planar maps (lecture notes).
\newblock Available at
  \url{https://www.math.u-psud.fr/~curien/cours/peccot.pdf}, 2017.

\bibitem{Curien-Le_Gall:Scaling_limits_for_the_peeling_process_on_random_maps}
{\sc Curien, N., and Le~Gall, J.-F.}
\newblock Scaling limits for the peeling process on random maps.
\newblock {\em Ann. Inst. Henri Poincar\'e Probab. Stat. 53}, 1 (2017),
  322--357.

\bibitem{Curien-Marzouk:sous_diff}
{\sc Curien, N., and Marzouk, C.}
\newblock Boltzmann planar maps are sub-diffusive.
\newblock \emph{In preparation}.

\bibitem{Curien-Miermont:Uniform_infinite_planar_quadrangulations_with_a_boundary}
{\sc Curien, N., and Miermont, G.}
\newblock Uniform infinite planar quadrangulations with a boundary.
\newblock {\em Random Struct. Alg. 47}, 1 (2015), 30--58.

\bibitem{Gwynne-Hutchcroft:preparation}
{\sc Gwynne, E., and Hutchcroft, T.}
\newblock Anomalous diffusion of random walk on random planar maps.
\newblock \emph{In preparation}.

\bibitem{Gwynne-Miller:Random_walk_on_random_planar_maps_spectral_dimension_resistance_and_displacement}
{\sc {Gwynne}, E., and {Miller}, J.}
\newblock {Random walk on random planar maps: spectral dimension, resistance,
  and displacement}.
\newblock {\em {\em Preprint available at
  \href{http://arxiv.org/abs/1711.00836}{\tt arXiv:1711.00836}}\/} (2017).

\bibitem{Jonasson-Schramm:On_the_cover_time_of_planar_graphs}
{\sc Jonasson, J., and Schramm, O.}
\newblock On the cover time of planar graphs.
\newblock {\em Electron. Comm. Probab. 5\/} (2000), 85--90.

\bibitem{Krikun-Local_structure_of_random_quadrangulations}
{\sc Krikun, M.}
\newblock {Local structure of random quadrangulations}.
\newblock {\em \href{https://arxiv.org/abs/math/0512304}{\tt
  arXiv:math/0512304}\/} (2005).

\bibitem{Le_Gall-Lehericy:Separating_cycles_and_isoperimetric_inequalities_in_the_UIPQ}
{\sc Le~Gall, J.-F., and Leh{\'e}ricy, T.}
\newblock {Separating cycles and isoperimetric inequalities in the uniform
  infinite planar quadrangulation}.
\newblock {\em {\em Preprint available at
  \href{http://arxiv.org/abs/1710.02990}{\tt arXiv:1710.02990}}\/} (2017).

\bibitem{Lee:Conformal_growth_rates_and_spectral_geometry_on_distributional_limits_of_graphs}
{\sc Lee, J.~R.}
\newblock {Conformal growth rates and spectral geometry on distributional
  limits of graphs}.
\newblock {\em {\em Preprint available at
  \href{http://arxiv.org/abs/1701.01598}{\tt arXiv:1701.01598}}\/} (2017).

\bibitem{Marckert-Miermont:Invariance_principles_for_random_bipartite_planar_maps}
{\sc Marckert, J.-F., and Miermont, G.}
\newblock Invariance principles for random bipartite planar maps.
\newblock {\em Ann. Probab. 35}, 5 (2007), 1642--1705.

\bibitem{Stephenson:Local_convergence_of_large_critical_multi_type_Galton_Watson_trees_and_applications_to_random_maps}
{\sc Stephenson, R.}
\newblock {Local convergence of large critical multi-type Galton-Watson trees
  and applications to random maps}.
\newblock {\em {\em To appear in} J. Theoret. Probab. {\em Preprint available
  at \href{http://arxiv.org/abs/1412.6911}{\tt arXiv:1412.6911}}\/} (2014).

\bibitem{Watabiki:Construction_of_non_critical_string_field_theory_by_transfer_matrix_formalism_in_dynamical_triangulation}
{\sc Watabiki, Y.}
\newblock Construction of non-critical string field theory by transfer matrix
  formalism in dynamical triangulation.
\newblock {\em Nuclear Phys. B 441}, 1-2 (1995), 119--163.

\end{thebibliography}

}

\end{document}